\documentclass[12pt,amscd]{amsart}
\usepackage[all]{xy}
\usepackage{graphicx}
\usepackage{amsmath,amsxtra,amssymb,latexsym, amscd,amsthm,enumerate}
\usepackage{indentfirst}
\usepackage[mathscr]{eucal}
\usepackage{stackrel}
\usepackage{float} 
\usepackage{algorithm2e}

\newtheorem{thm}{Theorem}[section]
\newtheorem{cor}[thm]{Corollary}
\newtheorem{lem}[thm]{Lemma}
\newtheorem{prop}[thm]{Proposition}

\theoremstyle{definition}

\newtheorem{exam}[thm]{Example}
\newtheorem{rem}[thm]{Remark}
\newtheorem{alg}[thm]{Algorithm}

\newcommand {\N} {{\mathbb N}}
\newcommand {\kk} {\Bbbk}
\DeclareMathOperator{\im}{im}
\DeclareMathOperator{\reg}{reg}
 \DeclareMathOperator{\iin}{in}
\DeclareMathOperator{\pd}{pd}

\begin{document}

\title[On the extremal Betti numbers of the binomial edge ideal  of closed graphs]{On the extremal Betti numbers of the binomial edge ideal  of closed graphs}

\author{Hern\'an de Alba}

\email{hdealbac@matematicas.reduaz.mx}
\address{CONACYT - Unidad Acad\'emica de Matem\'aticas de la Universidad Aut\'onoma de Zacatecas,Calzada Solidaridad
entronque Paseo de la Bufa, Zacatecas, Zac. 98000, Mexico}

\author{Do  Trong Hoang}

\email{dthoang@math.ac.vn}
\address{Institute of Mathematics, Vietnam Academy of Science and Technology, 18 Hoang Quoc Viet, 10307 Hanoi, Vietnam}
 
\subjclass[2010]{05C30, 05D15}
\keywords{Closed graphs, Binomial edge ideals, Projective dimension,  Betti numbers}
\thanks{}
\date{}
\dedicatory{}
\commby{}
\begin{abstract}  We study the equality of the extremal Betti numbers of the binomial edge ideal $J_G$ and those of its initial ideal ${\rm in}(J_G)$ for a closed graph $G$.  We prove that in some cases there is an unique extremal Betti number for ${\rm in}(J_G)$ and as a consequence there is an unique extremal Betti number for $J_G$ and these extremal Betti numbers are equal. 
\end{abstract}
\maketitle
\section*{Introduction}

Let $S = \kk[x_1,\ldots, x_n]$ be a polynomial ring over an arbitrary field $\kk$. If $I$ is a homogeneous ideal of $S$, then $I$ has an unique minimal graded free resolution up to isomorphism
  $$0\to \oplus_{j} S(-j)^{\beta_{l,j}(S/I)} \to  \oplus_{j} S(-j)^{\beta_{l-1,j}(S/I)} \to\cdots\to  \oplus_{j} S(-j)^{\beta_{0,j}(S/I)} \to S/I \to 0 $$
where $l \le n$, and $S(-j)$ is the $S$-module shifted by $j$. The number $\beta_{i,j}(S/I)$,  the {\it $ij$-th graded Betti number} of $S/I$, is an invariant of $S/I$.  The projective dimension of $S/I$  is defined to be $\pd(S/I):=\max\{i\mid \beta_{i,j}(S/I)\ne 0\}$. The regularity of $S/I$ is defined by $\reg(S/I):=\max\{j-i\mid \beta_{i,j}(S/I)\ne 0\}$.  A Betti number $\beta_{i,j}(S/I)\ne  0$   is called {\it extremal} if $\beta_{l,r} (S/I)= 0$ for all   $l\ge i, r\ge j+1$ and $r-l\ge j-i$.  A nice property of the extremal Betti numbers is $S/I$ has an unique extremal Betti number if and only if $\beta_{p,p+r}(S/I)\neq 0$, where $p=\pd (S/I)$ and $r=\reg(S/I)$.

Let $G$ be a  finite simple  graph   on the vertex set $V(G)=\{1,\ldots,n\}$ and edge set $E(G)$. Let $R:=\kk[x_1,\ldots,x_n,y_1,\ldots,y_n]$ be  a polynomial ring of $2n$ variables over a given field $\kk$.  The   {\it binomial edge ideal} of $G$ is  $$J_G=(x_iy_j-x_jy_i\mid \{i,j\}\in E(G) \text{ and } i<j)\subseteq R.$$ 
This ideal was independently introduced by   Herzog et al. \cite{HHHKR}; and  Ohtani \cite{O}.  Many of the algebraic properties and invariants of such ideals have been
studied in \cite{BN, CDI, Do, EHH, SZ}.

The  Gr\"obner basis with respect to the  lexicographic order induced by $x_1>\cdots> x_n>y_1>\cdots> y_n$   was computed in \cite[Theorem 1.1]{HHHKR}. It turned out that this Gr\"obner basis is quadratic if and only if the graph $G$ is a closed graph with respect to  the given labelling. We always have, by semicontinuity of Betti numbers, $\beta_{i,j}(R/J_G)\leq \beta_{i,j}(R/\iin(J_G))$ (see \cite[Corollary 3.3.3]{HH-B}); thus $\pd(R/J_G)\le \pd(R/\iin(J_G))$, and $\reg(R/J_G)\le \reg(R/\iin(J_G))$. When $G$ is a closed graph, Ene, Herzog and Hibi conjectured in \cite{EHH} that $\beta_{i,j}(R/J_G)=\beta_{i,j}(R/\iin(J_G))$ and they proved the conjecture in the case of $J_G$ is Cohen-Macaulay. In fact, they had a strong believe for the truthfulness of the conjecture in the case of the extremal Betti numbers. Later, in \cite{EZ}, Ene and Zarojanu showed that $\reg(R/J_G) = \reg(R/\iin(J_G))$ for a closed graph $G$. Recently, Baskoroputro proved in \cite{B} that $\beta_{i,j}(R/J_G)=\beta_{i,j}(R/\iin(J_G))$, when $G$ is a closed graph and $j=i+1$, moreover this equality is also true for any $i,j\in\N$ when $\reg(R/J_G)\leq 2$. In this paper we are interested to study the conjecture of Herzog, Hibi and Ene for the extremal Betti numbers.

Assume that   $G_1,\ldots, G_s$ are connected components of $G$. Let $R=\kk[x_j,y_j\mid j\in V(G)]$ and $R_k:=\kk[x_j,y_j\mid  j\in V(G_k)]$ for all $1\le k\le s$. Then   $R/J_G\cong \otimes_{k=1}^s R_k/J_{G_k}$ and $R/\iin(J_G)\cong \otimes_{k=1}^s R_k/\iin(J_{G_k})$.  If $\beta_{i,j}(R_k/J_{G_k})=\beta_{i,j}(R_k/\iin(J_{G_k}))$ for all $k$, then $\beta_{i,j}(R/J_{G})=\beta_{i,j}(R/\iin(J_{G}))$. Furthermore, assume that  $G=G_1\cup G_2$  and  $G_1\cap G_2=\{v\}$, where $v$ is a cut point of $G$ and $G_1, G_2$ are two induced subgraphs of $G$ without  cut point. Let $p_i:=\pd(R_i/\iin(J_{G_i}))$ and $r_i=\reg (R_i/\iin(J_{G_i}))$ for $i=1,2$. If $\beta_{p_i,p_i+r_i}(R_i/\iin(J_{G_i}))$ is the unique extremal Betti number of $R_i/\iin(J_{G_i})$ then $\beta_{p,p+r}(R/J_{G})=\beta_{p,p+r}(R/\iin (J_{G}))\ne 0$, where $p:=p_1+p_2$ and $r:=r_1+r_2$; and $\pd(R/J_{G}) = \pd(R/\iin(J_{G}))=p$ and $\reg(R/J_{G}) = \reg(R/\iin(J_{G}))=r$ (see Proposition \ref{depth-extremal}).  Therefore, we will deal with the case $G$  is a connected closed graph without  cut point.  We will see that in order to define $G$ is enough to define a vector $\mu(G)=(\mu_1,\ldots,\mu_{n})$, where $\mu_1,\ldots,\mu_n$ is a decreasing sequence of non-negative integers,   $\mu_{n-2} = \mu_{n-1}=\mu_n=0$ and $\mu_j\le n-2-j$ for all $1\le j\le n-3$  (see Lemmas \ref{lambda} and \ref{lem15}). For a connected closed graph $G$ without cut point, we will study the connected graph $H$ such that the edge ideal $I(H)$ of $H$ is equal to ${\iin}(J_G)$.  The graph $H$ will be called  an initial-closed graph.   
We will focus on  the projective dimension and the extremal Betti numbers of $H$ in order to obtain the main result of this paper:

\medskip

\noindent {\bf Theorem   \ref{main-thm3}. } {\it Let $G$ be a connected closed graph with $\ell$ cut points $v_1,\ldots,v_{\ell}$. Assume that $G=G_1\cup\ldots\cup  G_{\ell+1}$ such that $G_i\cap G_{i+1}=\{v_i\}$ and $G_i\cap G_j = \emptyset$ for $i=1,\ldots,\ell$ and $i\ne j\ne i+1$. Let  $n_i=|V(G_i)|$ and  $\mu(G_i)= (\mu_{i1},\ldots,\mu_{in_i})$, where $s_i:=\min\{k-1\mid \mu_{ik} =0\}$. If for each $i$,    one of three following conditions is satisfied:
\begin{enumerate}
\item $s_{i}=0$,   or 
\item $0<\mu_{is_i}=\ldots=\mu_{i1}$,   or 
\item $0<\mu_{is_i}<\ldots<\mu_{i1}<n_i-s_i$;
\end{enumerate}
 then $R/\iin(J_G)$ and $R/J_G$  have an unique extremal Betti number, and they are equal. In particular, $\pd(R/J_G)=\pd(R/\iin(J_G))$. }

\medskip

The paper is organized as follows. In Section 1, we recall some basic notations and the terminologies from Graph theory. In Section 2,  we investigate structure of initial-closed graphs, and give an upper bound for the projective dimension of the edge ideal of  such graphs. We give also  a characterization for  the Cohen-Macaulay property of the initial-closed graphs. In Section 3, we give an algorithm  that allows us to compute the Betti numbers of the edge ideal of the  initial-closed graphs (see Theorem \ref{Betti}) and we also prove that for some families of initial-closed graphs the extremal Betti numbers of its edge ideal are unique. As  a consequence we obtain that this lower bound for projective dimension of initial-closed graphs, and furthermore in some cases this bound  is sharp.  In the last section,  we obtain that the conjecture of Hibi, Herzog and Ene for the extremal Betti numbers of the binomial edge ideal of a closed graph and its initial ideal are equal in some cases, which is the main result of this paper.

 \section{Connected closed graphs without cut point} 
 
 We now recall some terminologies from graph theory (see \cite{BM}). Let $G$ be a simple graph on the vertex set $V(G)$ and edge set $E(G)$. An edge $e\in E(G)$ connecting two vertices $x$ and $y$ will be also written as $\{x,y\}$.  In this case, it is said that $x$ and $y$ are
adjacent. A matching in a graph is a set of edges, no two of which meet a common vertex. An induced matching $M$ in a graph $G$ is a matching where no two edges of $M$ are adjacented by an edge of $G$.  The maximum size of an induced matching in $G$ is denoted $\im(G)$. The neighborhood of  $x$ in $G$ is the set  $$N_G(x) := \{y\in V(G)  \mid \{x,y\}\in E(G) \text{ for some }x\in V(G)\},$$
the close neighborhood of $x$ is $N_G[x]:=N_G(x)\cup \{x\}$.  The number $\deg_G(x):=|N_G(x)|$ is called the {\it degree} of  $x$ in  $G$.  For a subset $S$ of $V(G)$, we denote by $G[S]$ the induced subgraph of $G$ on the vertex set $S$, and denote $G\backslash S$ by $G[V(G)\backslash S]$. For each $x\in V(G)$, we write $G\backslash x$ (resp. $G_x$) stands for $G\backslash \{x\}$ (resp. $G\backslash N_G[x]$).  The subset $S$ of $V(G)$ is called {\it clique} of $G$ if any two vertices in $S$  are adjacent.    A point $v$ is a {\it cut point} of a connected graph $G$  if  $G\backslash  v$ is disconnected.   

A simple graph G on the vertex set $V(G)=\{1,\ldots,n\}$ is called {\it closed with respect to the given labeling}  if the following condition is satisfied:   whenever $\{i,j\}$  and $\{i, k\}$ are edges of $G$ and either $i < j$, $i < k$ or $i > j, i > k$, then $\{ j, k\}$ is also an edge of $G$. One calls a graph $G$ is {\it closed} if it is closed with respect to some labeling of its vertices.  On this paper, for any  closed graph we will  fix  the labelling on $V(G)$ such that the graph $G$ is a closed graph with respect to this labeling.

Now let $G$ be a connected closed graph. We define  $N_G^{>}(i):=\{j\in V(G)\mid i<j, \text{ and } \{i, j\} \in E(G)\}$, and $N_G^{<}(i):=\{j\in V(G)\mid i>j, \text{ and } \{i, j\}\in E(G)\}$. We denote $\deg_G^{>}(i):=|N_G^{>}(i)|$ and $\deg_G^{<}(i):=|N_G^{<}(i)|$.  Thus, $\deg_G(i)= \deg_G^{>}(i) + \deg_G^{<}(i)$.  
  
\begin{lem}\label{neibour} Let $G$ be a connected closed graph. Then 
\begin{enumerate}
\item {\rm  \cite[Proposition 2.2]{CE}} $N_G^{>}(i)$ is a clique and  equal to $[i+1,i+r]$, where $r=\deg_G^{>}(i)$,  
\item  If  $N_G^{>}(j) =[j+1,k]$ and $\{i,t\}\in E(G)$  with $i<j<k$, then $t\le k$. 
\end{enumerate}
 \end{lem}
\begin{proof} Assume on the contrary that $t>k$. Since $\{i,t\}\in E(G)$, so by (1)   we have  $[i+1,\ldots,t]\subseteq N_G^{>}(i)$.  Thus $\{j, t\}\in E(G)$  because   $\{i, j\}, \{i,t\}\in E(G)$. This is a contradiction to  the assumption. 
\end{proof}

We associate to a closed graph $G$  a vector of  integers  $\mu(G) = (\mu_1,\ldots, \mu_n),$ where  $\mu_j:=n-j- \deg^{>}_G(j)$ for all $1\le j\le n$.  The  sequence of the  numbers $\mu_1,\ldots, \mu_n$ is a decreasing  sequence of non-negative integers  by the following lemma:

\begin{lem}     \label{lambda}  Let $G$ be a connected closed graph and $\mu(G)=(\mu_1,\ldots,\mu_{n})$. Then
    $0\le \mu_{i+1}\le \mu_i$ for all $1\le i\le n-1$. In particular, $\mu_n=\mu_{n-1}=0$. 
 \end{lem} 
\begin{proof}   For each $i$, by Lemma \ref{neibour}(1), $N_G^{>}(i)\subseteq [i+1,n]$. This yields $\deg_G^{>}(i)\le n-i$, and so $\mu_i\ge 0$.  By Lemma \ref{neibour}(1) again,  we assume $N_G^{>}(i+1) = [i+2,k]$. By Lemma \ref{neibour}(2),   $N_G^{>}(i) = [i+1,t]$, where $t\le k$. Thus   $\deg_G^{>}(i) = t-i\le k-i = \deg_G^{>}(i+1)+1.$ This means that $\mu_{i+1}\le \mu_i$.

Next in order to prove the last statement, it suffices to prove that $\mu_{n-1}=0$.     Indeed, since $G$ has no isolated vertices,  so $N_G^{<}(n)\ne \emptyset$. Thus, there exists   an edge  $\{t,n\}$ of $G$ with $t<n$. By Lemma \ref{neibour}(1),  $N_G^{>}(t)=[t+1,n]$ and $N_G^{>}(t)$ is a clique. Hence $\{n-1,n\}\in E(G)$, and thus $\deg_G^{>}(n-1)=1$. It implies that      $\mu_{n-1}=0$. 
 \end{proof}

 \begin{lem} \label{def-ini} Let   $G$ be a connected closed graph, and      $H'$  be a graph with edge ideal $\iin(J_G)$.  Then
\begin{enumerate}
\item $\deg_G^{>}(i) = \deg_{H'}(x_i)$ and $\deg_G^{<}(i) = \deg_{H'}(y_i)$ for all $i$. In particular,  $x_n$ and $ y_1$ are isolated vertices of $H'$.
\item   $H'$ is  a bipartite graph with  bipartition $(X,Y)$, where $X=\{x_1,\ldots,x_{n}\}$ and $Y=\{y_1,\ldots,y_n\}$, and    satisfies   three following conditions:
\begin{enumerate}
\item  if $\{x_i,y_j\}\in E(H')$, then $i<j$, 
\item  if  $\{x_i,y_j\}, \{x_i,y_k\}\in E(H')$ with  $i<j<k$, then  $\{x_j,y_k\}\in E(H')$, 
\item  if $\{x_i,y_k\},\{x_j,y_k\}\in E(H')$ with  $i<j<k$, then  $\{x_i,y_j\}\in E(H')$.
\end{enumerate}  
\item Each non-trivial connected component of $H'$  is a bipartite graph with   bipartition  $\{x_{i_1},\ldots, x_{i_{u-1}}\} \cup \{y_{i_2},\ldots,y_{i_{u}}\}$, where  $i_1<\ldots<i_u$.  
\item $H'\backslash \{x_n,y_1\}$ has no isolated vertices. 
\end{enumerate} 
  \end{lem}  
\begin{proof}  The  statements (1) and (2) follow from the definition of closed graph $G$. 

\medskip 

(3)  Since $H'$ is a bipartite graph, each non-trivial connected component of $H'$  is also a bipartite graph.  We assume its  bipartition is  $\{x_{i_1},\ldots, x_{i_{u-1}}\} \cup \{y_{j_1},\ldots,y_{j_{v-1}}\}$, where  $i_1<\ldots<i_{u-1}$ and $j_1<\ldots<j_{v-1}$. Since $H'$ satisfies the three conditions of (2), $v=u, \text{ and }    j_{k} = i_{k+1}, \text{  for all }  1\le k\le u-1$. 
 
\medskip 
 
(4)  Assume on the  contrary that  $x_j$ is an isolated vertex of $H'$ ($j\ne n$).  Since $G$ is connected, so there exists  $\{t,k\}\in E(G)$ such that   $t<j<k$. By Lemma \ref{neibour}(1), $[t+1,k]\subseteq  N_G^{>}(t)$ and   
$N_G^{>}(t)$ is a clique. Then $\{j,k\}\in E(G)$, which is a  contradiction.   

Similar to the proof of above argument, we conclude $y_2,\ldots,y_n$ are  not also isolated vertices of $H'\backslash \{x_n,y_1\}$, as required. 
\end{proof}

Let $G_1$ and $G_2$ be two graphs. We set $G:=G_1\cup G_2$  is  a  graph with $V(G) = V(G_1)\cup V(G_2)$ and $E(G) = E(G_1)\cup E(G_2)$.

\begin{lem} \label{def-ini-clo}
Let   $G$ be  a connected closed graph, and  $H'$  be a graph with edge ideal $\iin(J_G)$.  Then    $G$ has no cut point if and only if   $H'\backslash \{x_n,y_1\}$  is connected. 
\end{lem} 
\begin{proof} 
For the  sufficient part, assume  that there exists a cut point of $G$, say $v$. We may assume $G=G_1\cup  G_2$,  $V(G_1)\cap V(G_2)=\{v\}$ and  $E(G) = E(G_1)\cup E(G_2)$, where $G_1$ and $G_2$ are non-trivial subgraphs of $G$.  Note that $v\ne 1, n$ because $\deg_G^{>}(v)>0$ and $\deg_G^{<}(v)>0$.    Let $u_1\in N_{G_1}(v)$. Without loss of the generality, we assume $u_1<v$. We claim that   $\deg_{G_1}^{>}(v) =  \deg_{G_2}^{<}(v) = 0$. In fact,  let  $u_2\in N_{G_2}(v)$. If $u_2<v$, then $\{u_1, u_2\}\in E(G)$ because $\{u_1,v\}, \{u_2,v\}\in E(G)$ and $G$ is closed graph. It is impossible because $E(G)= E(G_1) \cup E(G_2)$. Thus, $u_2>v$, and so $u_2\in N_{G_2}^{>}(v)$. On the other hand, $N_{G_2}(v)=N_{G_2}^{>}(v)$, and so $N_{G_2}^{<}(v) = \emptyset$. This yields  $\deg_{G_2}^{<}(v) = 0$.  Simililar the above argument,  $\deg_{G_1}^{>}(v)  = 0$,  as claimed.

We set $H_1$ (resp. $H_2$) is  a  connected induced subgraph  of $H'$ containing   $y_v$ (resp. $x_v$).  By Lemma \ref{def-ini},  
$H_1$ (resp. $H_2$) is bipartite  with  bipartition $\{x_{i_1},\ldots,x_{i_{u-1}}\}\cup \{y_{i_2},\ldots,y_{i_{u}}\}$, where $i_1<\ldots<i_u$ (resp.   $\{x_{j_1},\ldots,x_{j_{v-1}}\}\cup \{y_{j_2},\ldots,y_{j_{v}}\}$,  where $j_1<\ldots<j_v$). Hence $\{i_1,\ldots,i_u\} \subseteq V(G_1)$ and  $\{j_1,\ldots,j_v\} \subseteq V(G_2)$.  By the above claim, we imply that  $i_u=v=j_1$. Thus,   $H_1$ and $H_2$ are connected components of $H'\backslash \{x_n,y_1\}$.

Now we prove  the necessary part. Suppose that  $H'\backslash \{x_n, y_1\}$ is disconnected. Then we may assume $H_1$ and   $H_2$ are two connected components of $H'\backslash \{x_n, y_1\}$.  By Lemma \ref{def-ini},  $H_1$ (resp. $H_2$) is bipartite with  bipartition  $\{x_{i_1},\ldots,x_{i_{u-1}}\}\cup \{y_{i_2},\ldots,y_{i_{u}}\}$ where  $i_1<\ldots<i_u$ (resp. $\{x_{j_1},\ldots,x_{j_{v-1}}\}\cup \{y_{j_2},\ldots,y_{j_{v}}\}$ where $j_1<\ldots<j_v$). We set $G_1$ (resp. $G_2$) is  an induced subgraph of $G$ on $\{i_1,\ldots,i_u\}$ (resp. $\{j_1,\ldots,j_v\}$). As $G$ is connected,   $V(G_1)\cap V(G_2)\ne \emptyset$. Then we may assume 
$i_1<\ldots<i_{u}=j_1<\ldots<j_{v}.$  This yields, $i_u$ is a cut point of $G$.
\end{proof}

\begin{lem} \label{lem15} If $G$ is a  connected closed graph without  cut point, then  $\mu_i\le n-i-2$ for all $1\le i\le n-2$. In particular,  $\mu_{n-2}=0$. 
\end{lem} 
\begin{proof} Let $H'$ be a graph with edge ideal $\iin(J_G)$. By Lemma \ref{def-ini-clo},  $H'\backslash \{x_n,y_1\}$ is a connected graph.  Thus,  $\deg_{H'}(x_{i})\ge 2$  for all $1\le i\le n-2$.  By 
Lemma \ref{def-ini}(1), $\deg_G^{>}(i)\ge 2$.  This implies  $\mu_i\le n-i-2$.
 \end{proof}

Let $G$ be a connected closed graph without cut point. Then the  connected graph $H:=H'\backslash \{x_n,y_1\}$ in   the assertion of Lemma \ref{def-ini-clo}  is so-called {\it initial-closed} graph. The such graph is a bipartite graph with   bipartition $(X,Y)$, where $X=\{x_1,\ldots,x_{n-1}\}$ and $Y=\{y_2,\ldots,y_{n}\}$ and  $n\ge 2$.  We associated to the initial-closed graph $H$ the  vector    $\mu(H):=(\mu_1,\ldots,\mu_{n-1}),$  where $\mu_i:=    n-i-\deg_H(x_i)$. By Lemmas \ref{lambda}, \ref{def-ini} and \ref{lem15},  $\mu_1\ge \ldots\ge \mu_{n-3}\ge \mu_{n-2} =  \mu_{n-1} =0$, and $\mu_i\le n-i-2$ for $1\le i\le n-3$.

\begin{exam}
The  graph $G$ in Figure \ref{fig1}  is a connected closed graph without cut point with $\mu(G) = (3,1,0,0,0,0)$, and  its initial-closed graph $H$ with $\mu(H)= (3,1,0,0,0)$.    
\begin{center}
\begin{figure}[H]
\begin{tabular}{cc}
 \includegraphics[scale=0.5]{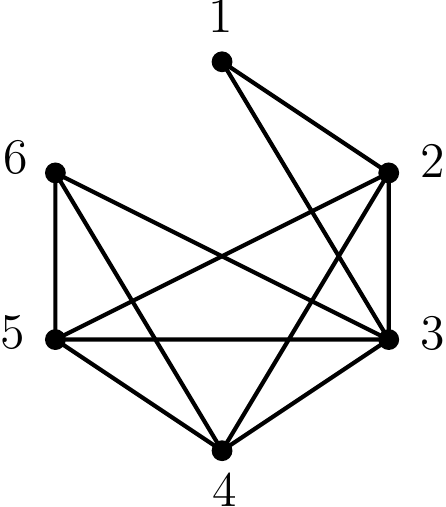}  \qquad    &   \qquad
 \includegraphics[scale=0.7]{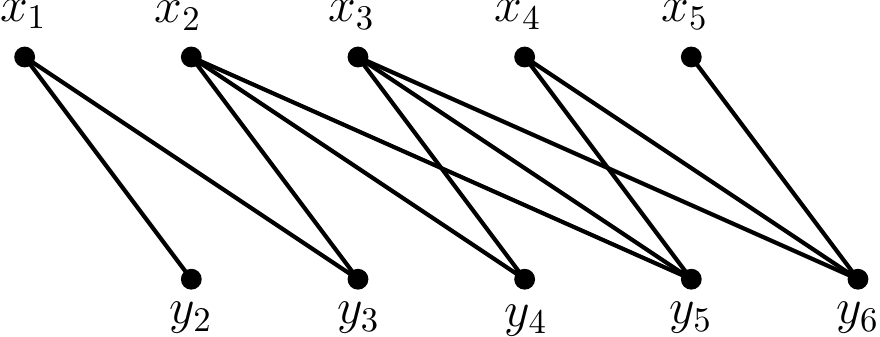} \\
 \end{tabular}
 \caption{Closed graph $G$ and its initial-closed graph $H$} \label{fig1}
 \end{figure} 
\end{center}
\end{exam} 
   
\begin{lem}\label{intr_depth}   Assume   $G=G_1\cup G_2$ is a connected closed graph and $V(G_1)\cap V(G_2)=\{v\}$, where $v$ is a  cut point of $G$, and  $G_1, G_2$ are subgraphs without cut point of $G$.  Let   $H'$ be  a graph with edge ideal $\iin(J_G)$, and let $H_1$ (resp.  $H_2$)  be an initial-closed graph of $G_1$ (resp.   $G_2$).  Then  the connected components of   $H'\backslash \{x_n,y_1\}$ are $H_1$ and $H_2$.   
 \end{lem}
\begin{proof}    By the assumption, $G_1$ and $G_2$ are also closed graphs.  From Lemmas \ref{def-ini} and  \ref{def-ini-clo}, $H_1$  (resp.  $H_2$) is  a connected bipartite graph with bipartition $\{x_{i_1},\ldots,x_{i_{k-1}}\}\cup \{y_{i_2},\ldots,y_{i_{k}}\}$, where $i_1<\ldots<i_k$ (resp. $\{x_{j_1},\ldots,x_{j_{l-1}}\}\cup \{y_{j_2},\ldots,y_{j_{l}}\}$, where $j_1<\ldots<j_l$).

Let $u_1\in N_{G_1}(v)$. Without  loss of the generality, we may assume $u_1<v$. Thus, $\deg_{G_1}^{>}(v)=\deg_{G_2}^{<}(v)=0$, and so  $i_k=v=j_1$. Moreover, since $V(G)= V(G_1)\cup V(G_2)$, so $i_1=1$ and $j_l=n$. Thus,  all  connected components of $H'\backslash \{x_n,y_1\}$ are $H_1$ and $H_2$.  
 \end{proof}

A simplicial complex $\Delta$ on the vertex set $V(\Delta) := \{1,\ldots ,n\}$ is a collection of subsets of $V(\Delta)$ such that $F\in \Delta$ whenever $F\subseteq F'$ for some
$F'\in \Delta$. Given  any field $\kk$, we define the Stanley-Reisner ideal  $I_{\Delta}$ in $\kk[V(\Delta)]:=\kk[x_1,\ldots,x_n]$ of $\Delta$ to be the squarefree monomial ideal
$$I_{\Delta}:= (x_{j_1}\ldots x_{j_s}\mid j_1<\ldots<j_s \text{ and } \{j_1,\ldots,j_s\}\in \Delta).$$

For a subset $W$ of $V(\Delta)$ the restriction of $\Delta$ on $W$ is the subcomplex $\Delta[W]:=\{F\in \Delta\mid F\subseteq W\}$.  We denote by $\widetilde{H}_j(\Delta; \kk)$  is reduced homology group of a simplicial complex $\Delta$ over $\kk$. A very useful result to compute the graded Betti numbers of the Stanley- Reisner ideal of simplicial complex is the so-called Hochster formula  (c.f. \cite[Theorem 8.1.1]{HH-B}). 
 
$$\beta_{i,j}(\kk[V(\Delta)]/I_{\Delta})= \sum_{W\subseteq V(\Delta), |W|=j} \dim_{\kk} \widetilde{H}_{j-i-1}(\Delta[W];\kk). $$

To each finite simple graph  $G$  with vertex set $V(G)=\{x_1,\ldots,x_n\}$ and  edge set $E(G)$, one associates the edge ideal  $I(G)$ of the  polynomial ring $\kk[V(G)]:=\kk[x_1,\ldots,x_n]$ which is generated by all monomials $x_ix_j$ such that $\{x_i, x_j\} \in E(G)$.   
Let $\Delta(G)$ be the set of all independent sets of $G$. Then, $\Delta(G)$ is a simplicial complex, called the independence complex of $G$. We can see that $I_{\Delta(G)} = I(G)$.  Note that $\Delta(G[W]) = \Delta(G)[W]$ for some  $W\subseteq V(G)$.  Therefore, Hochster formula  is also applied to compute Betti numbers of edge ideals. We write $\beta_{i,j}(G)$, $\pd(G)$,  and $\reg(G)$ as shorthand for  $\beta_{i,j}(\kk[V(G)]/I(G))$, $\pd(\kk[V(G)]/I(G))$,  and $\reg(\kk[V(G)]/I(G))$, respectively.

Let $\Delta_1$ and $\Delta_2$ be  simplicial complexes on the disjoint vertex sets $V_1$ and $V_2$, respectively. Define the {\it join}  on the vertex $V_1\cup V_2$ to be $\Delta_1*\Delta_2:=\{\sigma \cup \tau\mid \sigma \in \Delta_1, \tau\in \Delta_2\}$.  If $H_1$ and $H_2$ are two connected components of a graph $H$, then $\Delta(H) = \Delta(H_1)*\Delta(H_2)$.

\begin{prop} \label{depth-extremal}   Assume   $G=G_1\cup G_2$ is a connected closed graph and $V(G_1)\cap V(G_2)=\{v\}$, where $v$ is a  cut point of $G$, and  $G_1, G_2$ are subgraphs without cut point of $G$. Let $p_i:=\pd(R_i/\iin (J_{G_i}))$ and $r_i:=\reg(R_i/\iin (J_{G_i}))$ for  $i=1,2$,  where $R:=\kk[x_k,y_k \mid k\in V(G)]$ and $R_i:=\kk[x_k,y_k \mid k\in V(G_i)]$.   If   $\beta_{p_i,p_i+r_i}(R_i/\iin(J_{G_i}))\ne 0$ for all $i=1,2$, then  $\beta_{p,p+r}(R/J_{G})=\beta_{p,p+r}(R/\iin (J_{G}))\ne 0$, where $p=p_1+p_2$ and $r=r_1+r_2$. In particular, $\reg(R/J_G) = \reg(R/\iin(J_G))=r$ and $\pd(R/J_G) = \pd(R/\iin(J_G))=p$.   
\end{prop}

\begin{proof} We assume $H'$ is the graph with edge ideal $\iin(J_G)$, and let $H_1$ (resp.  $H_2$)  be an initial-closed graph of $G_1$ (resp.   $G_2$). Let $H:=H'\backslash \{x_n,y_1\}$ and $n_i=|V(H_i)|$. 

For each $i=1,2$, by the  assumption, $\beta_{p_i,p_i+r_i}(R_i/\iin (J_{G_i}))\ne 0$. We know that  the Hilbert function of $R_i/J_{G_i}$, $$H(R_i/J_{G_i};t)=\frac{\sum_{i,j}(-1)^{i}\beta_{i,j}(R_i/J_{G_i})t^j}{(1-t)^{n_i}},$$ is equal to the Hilbert function of $R_i/\iin(J_{G_i})$, $$H(R_i/(\iin J_{G_i});t)=\frac{\sum_{i,j}(-1)^{i}\beta_{i,j}(R_i/\iin(J_{G_i}))t^j}{(1-t)^{n_i}}.$$  
It implies that  $\beta_{p_i,p_i+r_i}(R_i/J_{G_i}) = \beta_{p_i,p_i+r_i}(R_i/\iin(J_{G_i})) = \beta_{p_i,p_i+r_i}(H_i)$. Thus, we have    $\pd(R_i/J_{G_i}) = \pd(R_i/\iin(J_{G_i}))=p_i=\pd(H_i)$, and $\reg(R_i/J_{G_i}) = \reg(R_i/\iin(J_{G_i}))=r_i=\reg(H_i)$. 
By Lemma \ref{intr_depth}, we have  $\reg(R/\iin(J_G))=r_1+r_2=r$ and $\pd(R/\iin(J_G))=p_1+p_2=p$.      

As $\beta_{p_i,p_i+r_i}(R_i/\iin(J_{G_i}))\ne 0$ and by Hochster formula, 
there exists a subset  $W_i$ of $V(H_i)$, $|W_i| = p_i+r_i$ such that $\dim_{\kk} \widetilde{H}_{r_i-1}(\Delta_i;\kk)> 0$, where $\Delta_i:=\Delta(H_i[W_i])$ for $i=1,2$. 

Now we let $W:=W_1\cup W_2 \subseteq V(H)$, and so $|W| =   p+r$. Since $H_1$ and $H_2$ are  two connected components of $H$,   $\Delta(H[W]) = \Delta_1*\Delta_2$.   By K\"unneth formula (c.f. \cite[Proposition 3.2]{BG}), we have   $\widetilde{H}_{r-1}(\Delta(H[W]);\kk)\cong \bigoplus_{i+j=r-2} \widetilde{H}_{i}(\Delta_1;\kk)\otimes \widetilde{H}_{j}(\Delta_2;\kk).$  It implies that 
 $ \dim_{\kk} \widetilde{H}_{r-1}(\Delta(H[W]);\kk)  \ge     \dim_{\kk} \widetilde{H}_{r_1-1}(\Delta_1;\kk)   \dim_{\kk}  \widetilde{H}_{r_2-1}(\Delta_2;\kk).$ 
Therefore, by Hochster formula, $\beta_{p,p+r}(H)\ne 0$, and so $\beta_{p,p+r}(R/\iin(J_G))\ne 0$. By equality of the Hilbert functions of $R/\iin(J_{G})$ and $R/J_{G}$, $\beta_{p,p+r}(R/J_G)\ne 0$. Thus  $\reg(R/J_{G}))= r=\reg(R/\iin(J_G))$ and $\pd(R/J_{G})= p= \pd (R/\iin(J_G))$.  
\end{proof}

\section{Upper bound for projective dimension}

In this section we will give an upper bound of the projective dimension of the edge ideal of some initial-closed graphs and for some specific cases we will obtain the exact value of the projective dimension of these ideals. In order to obtain  these results, the  following lemma will be very useful.

\begin{lem} {\rm \cite[Lemma  3.1]{DHS}} \label{Long} Let $x$ be a vertex of $G$ with neighbors $y_1, y_2, \ldots , y_m$. Then  
$$(I(G):x)  =  (I(G_x), y_1, y_2, \ldots , y_m), \text{ and  } 
(I(G),x)  =  (I(G\backslash x),x).$$
 \end{lem} 

The following lemma shall be  used a lot in this section.
  
\begin{lem}    \label{DepthLM}    Let  $x$ is a vertex of $G$. Then   
\begin{enumerate}
\item $\pd(G_x)+\deg_G(x)\le \max\{\pd(G), \pd(G\backslash x)\}$,
\item $\pd(G)\le \max\{\pd(G_x)+\deg_G(x), \pd(G\backslash x)+1\}$,
\item $1+ \pd(G\backslash x)\le \max \{\pd(G_x)+\deg_G(x)+1, \pd(G)\},$
\item If $1+\pd (G\backslash x) \le \pd(G_x) + \deg_G(x)$, then  $\pd(G)= \pd (G_x) + \deg_G(x),$ 
\item If $\pd (G_x) +\deg_G(x) < \pd(G\backslash x)$, then 
 $\pd(G)=  \pd(G\backslash x)+1.$ 
\end{enumerate} 
\end{lem}
\begin{proof} Let $S:=\kk[V(G)]$.  By Lemma \ref{Long},  we have $ \pd(S/(I(G):x))   =   \pd(G_x) +\deg_G(x)$,   and $\pd(S/(I(G),x))   =     \pd(G\backslash x) +1.$  The statements (1), (2) and (3) are followed by applying Depth lemma and Auslander-Buchsbaum formula for  the  following exact sequence: 
$$0\to S/(I(G):x) \overset{\cdot x}{\longrightarrow} S/I(G) \longrightarrow S/(I(G),x) \to 0.$$
Finally, (4) and (5) are consequences of (1), (2) and (3). 
\end{proof}
  
Following \cite{CN},  a Ferrers graph is a bipartite graph on two distinct vertex sets $X = \{x_1,\ldots , x_n\}$ and $Y = \{y_1, \ldots, y_m\}$ such that if $\{x_i,y_j\}$  is an edge of $G$, then so is $\{x_p,y_q\}$ for  $1 \le p \le i$ and $1 \le q \le j$. For any Ferrers graph $G$ there is an associated sequence of non-negative integers $\lambda(G) = (\lambda_1, \lambda_2, \ldots, \lambda_n)$, where $\lambda_i:=\deg_G(x_i)$. Notice that
the defining properties of a Ferrers graph imply that $\lambda_1 = m \ge \lambda_2 \ge \cdots \ge \lambda_n \ge 1$.

\begin{lem} {\rm \cite[Corollary 2.2]{CN}}\label{Ferrers} Let $G$ be a Ferrers graph with $\lambda(G)= (\lambda_1,\ldots,\lambda_n)$, and $I(G)$ be an edge ideal in   $\kk[V(G)]:=\kk[x_1,\ldots,x_n, y_1,\ldots,y_m]$. Then 
$$\pd(G) = \max_{1\le j\le n}\{\lambda_j+j-1\}.$$
\end{lem}
 
 Recall  $H$ is an initial-closed graph with its bipartition $\{x_1,\ldots,x_{n-1}\} \cup \{y_2,\ldots,y_n\}$, and   $\mu(H) = (\mu_1,\ldots,\mu_{n-1})$ is  an associated vector of $H$, where $\mu_i:=n-i-\deg_H(x_i)$ and   $\mu_1\ge \mu_2\ge \cdots\ge \mu_{n-3}\ge  \mu_{n-2} = \mu_{n-1}=0$,  furthermore $\mu_i\le n-2-i$ for all $1\le i\le n-3$.    From now on,   we  replace  $y_j$ by $y_{j-1}$ on the labelling of the vertex set of $H$. Then  the labelling on the  bipartition of $H$ would  be    $(X,Y)$,  where $X=\{x_1,\ldots, x_{n-1}\}$, $Y=\{y_1,\ldots,y_{n-1}\}$ and $n\ge 2$.   Therefore the edge ideal in $\kk[V(H)]:=  \kk[x_1,\ldots,x_{n-1}, y_1,\ldots,y_{n-1}]$ of the initial-closed graph $H$ is 
 $$I(H)=(x_iy_j \mid 1\le i\le n-1,  i\le j\le n-\mu_i -1).$$
 
 \begin{lem} \label{minus} Let  $H$ be  an initial-closed graph.  Then 
\begin{enumerate}
\item  The connected components of  $H\backslash \{x_i,y_{i}\}$ are also initial-closed graphs  for all $1\le i\le n-1$.
\item If $S=\{x_1,\ldots,x_i\}\cup \{y_1,\ldots,y_{i}\}$, for some $1\le i\le n-1$, then $H[S]$ is also an  initial-closed graph.
\end{enumerate}   
 \end{lem} 
  \begin{proof}   The proof follows immediately by the definition of the initial-closed graphs.
 \end{proof} 

\begin{exam} \label{exam1}
    A  graph  $H$ in Figure \ref{mu31} is an initial-closed graph with  $\mu(H) = (3,1,0,0,0)$.   
\begin{center}
\begin{figure}[H]
 \begin{tabular}{cc}
\includegraphics[scale=0.7]{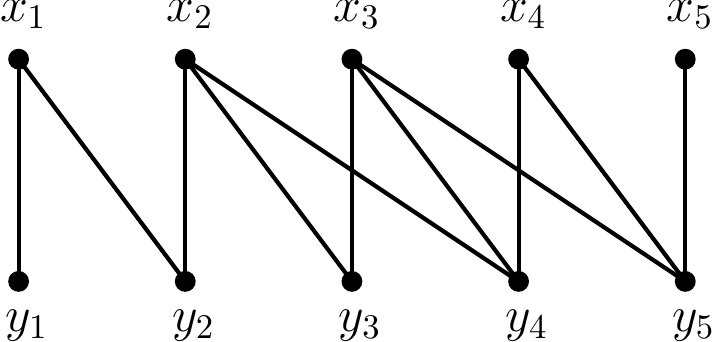} \qquad  &
\qquad  
\includegraphics[scale=0.55]{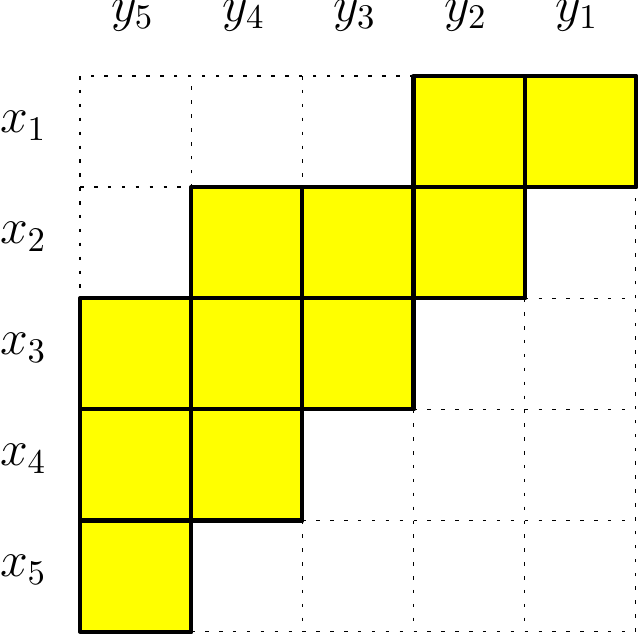} 
\end{tabular}  
\caption{Initial-closed graph $H$ and its illustration by diagram} 
\label{mu31}
\end{figure} 
  \end{center}
 \end{exam}

\begin{prop} \label{CM} 
Let   $H$ be an  initial-closed graph.   Then the  following conditions are equivalent:
\begin{enumerate}
\item $H$ is Cohen-Macaulay  (i.e. $\pd(H) = n-1$),
\item $H$ is unmixed,
\item $\mu(H)=(0,\ldots,0)$,
\item $\reg(H)=1$.
\end{enumerate}
\end{prop}
\begin{proof}  (1) $\Rightarrow$  (2)  is well-known.  
 \medskip

 (2) $\Rightarrow$  (3):  We have  $H$ is a connected bipartite graph with   bipartition $(X,Y)$, where $X=\{x_1,\ldots,x_{n-1}\}$ and $Y=\{y_1,\ldots, y_{n-1}\}$. Since $H$ is an initial-closed graph, $H$ satisfies two following conditions: 
 \begin{eqnarray*}
\text{(a)} && x_iy_i\in E(H) \text{ for all } 1\le i\le n-1,\\ 
\text{(b)}&&  \text{If }  x_iy_j\in E(H), \text{ then } i\le j.
 \end{eqnarray*} 
By Lemma \ref{neibour}(1) and Lemma \ref{def-ini}(1),  we assume $N_H(x_1)=\{y_1,\ldots,y_t\}$. If $t<n-1$, then $\{x_t,y_{t+1}\}\in E(H)$ because $H$ is connected. 
  By  \cite[Theorem 1.1]{V},   $\{x_1,y_{t+1}\}\in E(H)$, a contradiction. Hence $t=n-1$ which implies $\mu_1=0$. By Lemma \ref{lambda}, $\mu_i=0$ for all $1\le i\le n-1$.  

  \medskip

 (3) $\Rightarrow$  (4):  In this case $H$ is a Ferrers graph with $\lambda(H) = (n-1,\ldots,1)$. By \cite[Corollary 2.2]{CN},  $\reg(H)=1$.

  \medskip

 (4) $\Rightarrow$  (1):   By \cite[Theorem 2.2]{EZ}, $\reg(H) = \im(H)=1$. Hence,  $H$ is a Ferrers graph with $\lambda(H)= (n-1,\ldots, 2, 1)$. Thus the assertion follows from  \cite[Corollary 2.8]{CN}.
 \end{proof}

From now,  we will consider  the initial-closed graph $H$ with an associated vector  $\mu(H) = (\mu_1,\ldots,\mu_s,0,\ldots,0)\in \mathbb N^{n-1}$, where $0<\mu_s\le \ldots\le \mu_1$, $1\le s\le n-3$, and $\mu_j\le n-2-j$ for all $j=1,\ldots,s$. 
 
 \begin{lem} \label{lem1} Let $H$ be an initial-closed graph.  Then
\begin{enumerate} 
\item  If $s=1$, then  $\pd(H)=2(n-1)-(\mu_1+1),$  
\item  If $\mu_1=\cdots=\mu_s=1$, then  $\pd(H)=2(n-1)-(s+1).$  
 \end{enumerate} 
 \end{lem}
\begin{proof}  The assertion (2) is proved   similarly as the assertion  (1).  We now prove assertion  (1) by  induction on $n$.   If $n=4$, then $\mu(H)= (1,0,0)$. In this case, $H$ is a path of length $5$. By \cite[Corollary 7.7.35]{J},   $\pd(H) = 4$.  

We now assume that $n\ge 5$. By Lemma \ref{minus}, $H\backslash \{x_{n-1},y_{n-1}\}$ is an initial-closed graph with $\mu(H\backslash\{x_{n-1}, y_{n-1}\})=(\mu_1-1,0,\ldots,0)\in \mathbb N^{n-2}$. By the induction hypothesis,    $\pd(H\backslash \{x_{n-1},y_{n-1}\})  = 2(n-2)- \mu_1$.  Since  $N_H(x_{n-1})=\{y_{n-1}\}$,   $x_{n-1}$ is an isolated vertex of $H\backslash y_{n-1}$. Thus,   $  \pd (H\backslash y_{n-1})    =   \pd (H\backslash \{x_{n-1},y_{n-1}\})  = 2(n-2)- \mu_1.$ 

By the assumption,   $N_H(y_{n-1}) = \{x_2,\ldots,x_{n-1}\}$. Hence $\deg_H(y_{n-1})= n-2$, and    $V(H_{y_{n-1}}) = \{x_1,y_1,\ldots,y_{n-2}\}$. Thus,   $H_{y_{n-1}}$ is the  disjoint union of the  star graph on vertex set $\{x_1,y_1,\ldots,y_{n-1-\mu_1}\}$, which apex is  $x_1$, and the  isolated vertices $y_{n-\mu_1},\ldots,y_{n-2}$.  By \cite[Theorem 5.4.11]{J},       $\pd(H_{y_{n-1}}) = n-1-\mu_1$. 
  
By Lemma \ref{DepthLM}(4), we conclude that    $\pd(H) = 2(n-1)-(\mu_1+1)$, as required. 
  \end{proof}
 
 \begin{prop} \label{equal_mu} 
Let $H$ be an initial-closed graph. If  $\mu_1=\ldots=\mu_s=c\ge 1$, then $\pd(H) = 2(n-1)-(c+s).$  
\end{prop} 
\begin{proof} We prove the lemma by induction on $c$. If $c=1$, then    proposition follows from Lemma \ref{lem1}(2).  

We now assume that  $c\ge 2$. By the assumption,   $N_H(y_{n-1}) = \{x_{s+1},\ldots,x_{n-1}\}$.  Hence $\deg_H(y_{n-1}) = n-s-1$ and $V(H_{y_{n-1}}) = \{x_1,\ldots,x_s\}\cup \{y_1,\ldots,y_{n-2}\}$. Thus,  $H_{y_{n-1}}$ is the disjoint  union of the isolated vertices   $y_{n-c},\ldots,y_{n-2}$, and the induced graph $H'$ of $H$ on vertex set $\{x_1,\ldots,x_s\}\cup \{y_1,\ldots,y_{n-c-1}\}$.  Note that  $H'$ is  a  Ferrers graph with   $\lambda(H')=(n-c-1, n-c-2,\ldots, n-c-s)\in \mathbb N^s$. By Lemma \ref{Ferrers},     $\pd(H')  =n-c-1$, and so $\pd(H_{y_{n-1}})  = n-c-1$.

On the other hand, by Lemma \ref{minus}, $H\backslash\{x_{n-1}, y_{n-1}\}$ is also an initial-closed graph with $\mu(H\backslash\{x_{n-1}, y_{n-1}\}) = (c-1,\ldots,c-1,0,\ldots,0)\in \mathbb N^{n-2}$.  By the induction hypothesis, we have  $\pd(H\backslash \{x_{n-1},y_{n-1}\})  =  2(n-2) - (c-1+s)$.  
Since $x_{n-1}$ is an isolated vertex of $H\backslash y_{n-1}$, we obtain
 $\pd(H\backslash y_{n-1})  =    \pd(H\backslash \{x_{n-1},y_{n-1}\}) 
  =    2n-3 -(c+s).$ 
  
 By Lemma \ref{DepthLM}(4),  we conclude  that  $\pd(H) =2(n-1)-(c+s).$ 
\end{proof}

\begin{lem} \label{mu1_1}
Let $H$ be an initial-closed graph. If $s\ge 2$, then
$$\pd(H_{x_s}) =  n-1-s+\pd(H'),$$
where $H'$  is the induced graph of $H$  on the vertex set   $\{x_1,\ldots,x_{s-1}\}\cup \{y_1,\ldots,y_{s-1}\}$. Moreover $H'$ is an initial-closed graph and $\mu(H') = (\mu'_1, \ldots, \mu'_{s-3},0,0)$ with $\mu'_j=\max\{0, \mu_j-(n-s)\}$ for  $1\le j\le s-3$. In particular, 
$\pd (H_{x_s}) \ge  n-2$, and the equality holds   if and only if $\mu_1\le n-s$. 
\end{lem}  
\begin{proof}   By the assumption,  $N_H(x_{s})= \{y_{s}, \ldots, y_{n-\mu_s-1}\}$, and so $H_{x_s}$ is the disjoint union of two subgraphs $H'$ and $H''$, where $H'$ (resp. $H''$) is an  induced subgraph of $H$ on vertex set $\{x_1,\ldots,x_{s-1}, y_1,\ldots,y_{s-1}\}$ (resp. $\{x_{s+1},\ldots, x_{n-1}, y_{n-\mu_s},\ldots,y_{n-1}\}$). Then we get $$\pd(H_{x_s})= \pd(H') +\pd(H'').$$

We know that $H''$ is a Ferrers graph with  $\lambda(H'')=(\lambda''_{1},\ldots,\lambda''_{n-s-1})$, where $\lambda_j'':=\deg_{H''}(x_{s+j})$ and $\lambda''_j= \min\{\mu_s, n-s-j\}$ for $1\le j\le n-s-1$. By  Lemma \ref{Ferrers}, we get    $  \pd (H'')  =    \max_{1\le j\le n-s-1}\{\lambda_{j}'' + j- 1\}  =  n-s-1.$ In fact, we always have  $\pd(H'') = \max\{n-s-1,\max_{1\le j\le n-s-2}\{\lambda''_j+j-1\}\}.$ 
Since     $\lambda_j''+j-1=\min\{\mu_s,n-s-j\} + j- 1\le n-s-1$ for all 
   $1\le j\le n-s-1$. Thus,  $\pd(H'') =  n-s-1$.

Moreover, by Lemma  \ref{minus},  $H'$ is  also  an  initial-closed graph with  $\mu(H') = (\mu_1',\ldots,\mu_{s-1}')$, where   $\mu'_j=\max\{0, \mu_j-(n-s)\}$ for all $1\le j\le s-1$.  Furthermore, we obtain  $\mu_{1}' \ge \cdots \ge \mu_{s-3}'\ge \mu_{s-2}' =  \mu_{s-1}' = 0$.   

We now  prove the last assertion. We always have  $ \pd(H')\ge s-1$, and thus $\pd(H_{x_s})\ge n-2$. The equality holds  if and only if   $\pd(H') = s-1$. From Proposition \ref{CM}, $H'$ is a Cohen-Macaulay graph. By Proposition \ref{CM}, $\mu(H')=(0,\ldots,0)$. It means that $\mu_1\le n-s$.
\end{proof}

\begin{lem} \label{mu2} Let $H$ be  an initial-closed graph.  If $s\ge 2$ and $\mu_1< n-s$, then 
$$\pd(H\backslash x_s) = \begin{cases} n-1, &\text{ if } \mu_1=\ldots=\mu_{s-1}=n-s-1\\
   1+\pd(H\backslash \{x_s, y_s\}), &\text{ otherwise. }\end{cases}
$$
 \end{lem} 
\begin{proof}  Let $K:=H\backslash x_s$.  In order to prove this lemma, we divide the proof in two following cases:

\medskip
\noindent {\it Case 1:}  $\mu_{s-1}=n-s-1$. We have   $\mu_{s-1}\le \ldots\le \mu_1 <n-s$, and so $\mu_1=\ldots=\mu_{s-1}=n-s-1$. Hence $K$ is the  disjoint union of two Ferrers graphs $H'$ and $H''$, where $H'$ (resp. $H''$) is the  induced subgraph of $H$ on  $\{x_{1},\ldots, x_{s-1}, y_1,\ldots,y_s\}$ (resp. $\{x_{s+1},\ldots,x_{n-1},y_{s+1},\ldots,y_{n-1}\}$) with $\lambda(H') = (s,\ldots,2)\in \mathbb N^{s-1}$ (resp. $\lambda(H'')=(n-s-1,\ldots,1)\in \mathbb N^{n-s-1}$).   By Lemma \ref{Ferrers}, we obtain
 $$\pd(K) = \pd(H') + \pd(H'') = s +(n-s-1) =n-1.$$

\medskip
\noindent {\it Case 2:}  $\mu_{s-1}<n-s-1$. Thus, $H\backslash \{x_s,y_s\}$ is connected, and by Lemma \ref{minus}, $H\backslash \{x_s,y_s\}$ is also an  initial-closed graph. Note that $K\backslash y_s = H\backslash \{x_s,y_s\}$. Since $s\ge 2$ and $\mu_1<n-s$,  so  $\mu(K\backslash  y_s)= (\mu_1,\ldots,\mu_{s-1},0,\ldots,0)\in \mathbb N^{n-2}$. By Proposition \ref{CM},  $K\backslash  y_s$  is not Cohen-Macaulay.   Hence $\pd(K\backslash  y_s)>n-2$. 

Moreover, since   $\mu_1< n-s$,    $N_K(y_s) = \{x_1,\ldots,x_{s-1}\}$ and so $\deg_K(y_s)=s-1$.  Then $K_{y_s}$ is the union of  the isolated vertices $y_1,\ldots,y_{s-1}$ and a graph $K'$, where $K'$ is the  induced subgraph of $H$ on vertex set  $\{x_{s+1},\ldots,x_{n-1}, y_{s+1},\ldots,y_{n-1}\}$. Since $H$ is an initial-closed graph,   $K'$ is also an initial-closed  graph with  $\mu(K') = (0,\ldots,0)\in \mathbb N^{n-s-1}$. By Proposition \ref{CM},  $K'$ is Cohen-Macaulay and thus $\pd(K')=n-s-1$. Therefore, 
$\pd(K_{y_s})  = \pd(K')= n-s-1$.

By Lemma \ref{DepthLM}(5),  $\pd(K) = 1+\pd(K\backslash  y_s)$. On the other hand,  $\pd(H\backslash x_s) =   1+ \pd(H\backslash \{x_s, y_s\})$, which completes the proof of this  lemma.
 \end{proof}

\begin{thm} \label{main} Let $H$ be an  initial-closed graph. If $\mu_1< n-s$, then $$\pd(H) \le \max_{1\le j\le s}\{2(n-1)- (\mu_j+j)\}.$$
 In particular, if $\mu_s<\ldots<\mu_1<n-s$, then  $\pd(H) = 2(n-1)-(\mu_s+s)$.
 \end{thm} 
\begin{proof} We prove by   induction on $s$. If  $s=1$,   by Proposition \ref{equal_mu},  $\pd(H)=2(n-1)-(\mu_1+1)$.  Now we assume that   $s\ge 2$.  From Lemma \ref{mu1_1} and the assumption,  $\pd(H_{x_s}) =  n-2$. Note that  $N_H(x_s) = \{y_s,\ldots,y_{n-1-\mu_s}\}$, and so $\deg_H(x_s)=n-\mu_s-s$. 

Next we consider two following cases: 

\medskip
\noindent{\it Case 1:} $\mu_{s-1}=n-s-1$. By the assumption, $\mu_1=\ldots=\mu_{s-1}=n-s-1$. From  Lemma \ref{mu2}, we have $\pd(H\backslash x_s) = n-1$.  Hence, by Lemma \ref{DepthLM}(2), we have $ \pd(H)\le \max\{2(n-1)- (\mu_s+s), n\} = 2(n-1)- (\mu_s+s)\le \max_{1\le j\le s}\{2(n-1)- (\mu_j+j)\}.$

\medskip
\noindent{\it Case 2:} $\mu_{s-1}<n-s-1$. By Lemma \ref{mu2},    $\pd(H\backslash x_s) =1+\pd(H\backslash \{x_s,y_s\})$, and  $H\backslash \{x_s,y_s\}$ is also an  initial-closed graph with $\mu(H\backslash \{x_s,y_s\})= (\mu_1,\ldots,\mu_{s-1},0,\ldots,0)\in \mathbb N^{n-2}$. By the induction hypothesis,  $ \pd(H\backslash \{x_s,y_s\})  \le \max_{1\le j\le s-1}\{2(n-2)-(\mu_j+j)\}.$ 
From  Lemma \ref{DepthLM}(2), we have    $\pd(H) \le   \max\{2(n-1)-(\mu_s+s),\max_{1\le j\le s-1}\{2(n-2)-(\mu_j+j)\}+2\}  =  \max_{1\le j\le s}\{2(n-1)-(\mu_j+j)\}.
$ 

We shall now prove the last statement of this theorem. Assume  $\mu_s<\ldots<\mu_1<n-s$. Then  $\mu_s+s \le \ldots\le \mu_1+1$, and so $ \max_{1\le j\le s}\{2(n-1)-(\mu_j+j)\}=2(n-2)-(\mu_s+s)$. From Lemma \ref{mu2} and the above assertion,  we have 
\begin{eqnarray*}
\pd(H\backslash x_s)&=& 1+ \pd(H\backslash \{x_s,y_s\})\\
&\le& 1+ \max_{1\le j\le s-1} \{2(n-2)-(\mu_j+j)\} =  1+ 2(n-2)- (\mu_{s-1}+(s-1))\\
&\le& -1+ 2(n-1)- (\mu_s+s) = -1+ \pd(H_{x_s})+\deg_H(x_s).
\end{eqnarray*}
  By Lemma \ref{DepthLM}(4),   $\pd(H) =2(n-1)-(\mu_s+s)$, and the proof of theorem is complete.  
  \end{proof}

\section{Lower bound for projective dimension} 

The purpose of this section is to calculate Betti numbers of initial-closed graphs $H$, so that we give an algorithm (see Algorithm \ref{Algorithm}) using the biadjacency matrix of $H$. This algorithm is distinct to the algorithm given in \cite[section 2.4]{NR}. With this algorithm, we can obtain in some cases an explicit formula for the projective dimension of $R/I(H)$ in despite of the algorithm in \cite{NR}, where the formula obtained for the projective dimension is not explicit as in \cite[Proposition 2.26]{NR}.

Let $G$ be a bipartite graph with   bipartition $\{x_1,\ldots,x_n\}\cup \{y_1,\ldots,y_m\}$. The biadjacency matrix of $G$,  $M(G)= (a_{i,j}) \in \mathcal M_{n,m}(\{0,1\})$, is defined by $a_{i,j}=1$ if $\{x_i,y_j\}\in E(G)$, 0 otherwise. 
 
\begin{lem} { \rm (\cite[Lemma 1.6]{FG})}  \label{matrix}  Let $G$ be a bipartite graph, biadjacency matrix $M=M(G) = (a_{i,j}) \in \mathcal M_{n,m}(\{0,1\})$.   
\begin{enumerate}
\item If $M$ has a row or a column whose entries are all $0$, then $\widetilde{H}_i(\Delta(G);\kk)={\bf 0}$   for all $i\ge 0$.
\item If there exists $r$ and $c$ such that $a_{r,c}=1$ and the rest of entries on the row $r$ and the column $c$ are zeros,   $\widetilde{H}_i(\Delta(G);\kk) \cong \widetilde{H}_{i-1}(\Delta(G\backslash \{x_r,y_c\});\kk)$, for all $i>0$. 
\item If $M$ has two rows $r$ and $r'$ (resp. two columns $c$ and $c'$) such that $\{j: a_{r,j}=0\}\subset \{j: a_{r',j}=0\}$ (resp.  $\{i: a_{i,c}=0\}\subset \{j: a_{i,c'}=0\}$), then 
 $\widetilde{H}_i(\Delta(G);\kk) = \widetilde{H}_i(\Delta(G\backslash x_r);\kk) \text{ (resp. $\widetilde{H}_i(\Delta(G);\kk) = \widetilde{H}_i(\Delta(G\backslash y_c);\kk)$)},$    for all $i\ge 0$.
\end{enumerate} 
\end{lem} 

Let $H$ be a  bipartite graph with bipartition $\{x_1,\ldots,x_n\}\cup \{y_1,\ldots,y_m\}$.  Let $\lambda=(\lambda_1,\ldots,\lambda_n)$ be a partition and let $\mu=(\mu_1,\ldots,\mu_n)$  be a vector such that $\mu_1\ge \ldots\ge \mu_n\ge 0$,   $\mu_i\le \lambda_i$ for all $1\le i\le n$, and $\lambda_1=m$.  Then we call  $H$  {\it  skew Ferrers graph} if  its edge ideal is  
$$I(H) = (x_iy_{m-\mu_i},\ldots,x_iy_{m-\lambda_i+1}\mid 1\le i\le n).$$
  The skew Ferrers graphs have a long tradition in combinatorics according to skew Ferrers diagrams, see for example \cite{DF,MA}. Note that if $m=n$, $\mu_i\le n-1-i$ for $1\le i\le n-2$, $\mu_{n-1}= \mu_n=0$ and $\lambda_i=n+1-i$ for all $i$, then   $H$ is an initial-closed graph. If $\mu_i=0$ for all $i$, then $H$ is a Ferrers graph. 
  
\begin{alg} \label{Algorithm}   \textit{Input:}   Let $H$ be a  skew Ferrers graph. \\ 
\textit{Output:} An induced matching  $U$ of $H$, and a subset $S$ of $V(H)$ such that   $\widetilde{H}_i(\Delta(H);\kk) ={\bf 0} \text{ if } S\ne \emptyset$, and 
$$\widetilde{H}_i(\Delta(H\backslash S);\kk)  \cong \begin{cases}   \kk &\text{ if } i=|U|-1 \\ {\bf 0} &\text{otherwise.}  \end{cases} $$ 

\medskip

\noindent  Let $U:=\emptyset$ and $S:=\emptyset$.   

\begin{algorithm}[H]
 \While{$H \ne \emptyset$}{
We assume  bipartition of $H$ is $\{x_{i_1},\ldots,x_{i_u}\}\cup \{y_{j_1},\ldots,y_{j_v}\}$, where $i_1<\ldots<i_u,$   and  $j_1<\ldots<j_v$\;  
$U:=U\cup \{\{x_{i_u}, y_{j_v}\}\}$\;
 $S:=S\cup \{y_j\in V(H\backslash N_H(\{x_{i_u}, y_{j_v}\})) \mid$  the columm  with labelling  $y_j$   in $M(H\backslash N_H(\{x_{i_u}, y_{j_v}\}))$   is zero\}\;  
$H:= H\backslash (S\cup  N_H(\{x_{i_u}, y_{j_v}\}))$\; 
 } 
 \end{algorithm}
\noindent \textbf{return}($U, S$) 
 \end{alg}
\begin{proof} For each loop of the algorithm, we claim  that 
\begin{eqnarray}\label{eqn1} 
\widetilde{H}_i(\Delta(H);\kk) &\cong&  \widetilde{H}_{i-1}(\Delta(H\backslash N_H(\{x_{i_u},y_{j_v}\}));\kk),\\
\widetilde{H}_i(\Delta(H\backslash S);\kk) &\cong& \widetilde{H}_{i-1}(\Delta(H\backslash (S\cup  N_H(\{x_{i_u},y_{j_v}\})));\kk). \label{eqn2} 
\end{eqnarray}
Indeed,  let  $\alpha\in N_H(\{x_{i_u}, y_{j_v}\})\backslash \{x_{i_u}, y_{j_v}\}$, and so $\alpha=x_r$ or $\alpha = y_c$ for some $r\ne i_u$ and $c\ne j_v$. Without loss of generality, we may assume  $\alpha=x_r$, where $r\ne i_u$.  Since $H$ is a skew Ferrers graph,   $\{j \mid a_{r,j}=0\} \subseteq \{j\mid a_{i_u,j}=0\}$.  By Lemma \ref{matrix}(3), we have  $\widetilde{H}_i(\Delta(H);\kk)\cong \widetilde{H}_i(\Delta(H\backslash x_r);\kk).$ Repeating this  process for all $\alpha\in N_H(\{x_{i_u}, y_{j_v}\})\backslash \{x_{i_u}, y_{j_v}\}$, we obtain 
 $$\widetilde{H}_i(\Delta(H);\kk)\cong \widetilde{H}_i(\Delta(H\backslash (N_H(\{x_{i_u}, y_{j_v}\})\backslash \{x_{i_u}, y_{j_v}\}));\kk).$$  
By Lemma \ref{matrix}(2),   $\widetilde{H}_i(\Delta(H);\kk) \cong  \widetilde{H}_{i-1}(\Delta(H\backslash N_H(\{x_{i_u},y_{j_v}\}));\kk)$,  as the first assertion.

Let $U:=\{x_{i_u},y_{j_v}\}$, and  $S$ be a set containing $y_j\in V(H\backslash N_H(\{x_{i_u},y_{j_v}\}))$ such that  the  columm with labelling $y_j$ in $M(H\backslash N_H(\{x_{i_u},y_{j_v}\}))$  is zero.  Note that if  $S\ne \emptyset$, then by Lemma \ref{matrix}(1) and equality (1),   $\widetilde{H}_i(\Delta(H);\kk)=0$. Furthermore, $\widetilde{H}_i(\Delta(H\backslash S);\kk) \cong \widetilde{H}_{i-1}(\Delta(H\backslash (S\cup N_H(\{x_{i_u},y_{j_v}\})));\kk)$, which completes the proof of   the second claim.

Repeating the above loop $|U|$ times, by equality (\ref{eqn1}),     $\widetilde{H}_i(\Delta(H);\kk) =0$ if $S\ne \emptyset$; and furthermore by equality  (\ref{eqn2})  we have   $$\widetilde{H}_i(\Delta(H\backslash S);\kk)  \cong   \cdots  \cong   \widetilde{H}_{i-|U|}(\{\emptyset\};\kk) = \begin{cases} \kk &\text{ if } i=|U|-1\\ 
 {\bf 0} &\text{otherwise.}\end{cases}
$$  \end{proof}

\begin{rem} \label{main-thm2}
When we apply this algorithm to a skew Ferrers graph $H$ and using Hochster formula we get $\beta_{i,j}(H)$ is independent on the  base field for all $i,j$. This result was obtained also by Nagel and Reiner in \cite[Corollary 2.16]{NR}. 
\end{rem} 
 
Let $H$ be a skew Ferrers graph.  By Algorithm \ref{Algorithm}, the edge set of $H$ can be partitioned into subsets as follows:
$E(H) = \bigcup_{e\in U} E_{e},$
where $E_{\{x_i,y_j\}} = \{\{x_k,y_l\}\in E(H)\mid k\le i, l\le j, \text{ and }
 \{x_k,y_l\}\cap N_H(\{x_i,y_j\})\ne \emptyset\}$. This partition is called {\it rectangular decomposition}, which  is similar to the rectangular decomposition given in \cite[Section 2.4]{NR}.
 
 \begin{exam}  A skew Ferrers diagram in Figure \ref{fig3} with  $\mu=(4,2,1,1,0,0)$ and $\lambda=(6,5,4,4,2,1)$. Applying Algorithm \ref{Algorithm},   $U=\{\{x_6,y_6\}, \{x_4,y_5\}, \{x_2,y_2\}\}$ and $S=\{y_1\}$. Then $E_{\{x_2,y_2\}} = \{\{x_1,y_1\}, \{x_1,y_2\},  \{x_2,y_2\} \}$, $E_{\{x_6,y_6\}} = \{\{x_5,y_5\}, \{x_5,y_6\},  \{x_6,y_6\} \}$, and 
 $E_{\{x_4,y_5\}} = \{\{x_2,y_3\}, \{x_2,y_4\}, \{x_3,y_3\},  \{x_3,y_4\},  \{x_3,y_5\},   \{x_4,y_3\}, \{x_4,y_4\}, \{x_4,y_5\} \}$. 
\begin{center}
\begin{figure}[H]
 \includegraphics[scale=0.5]{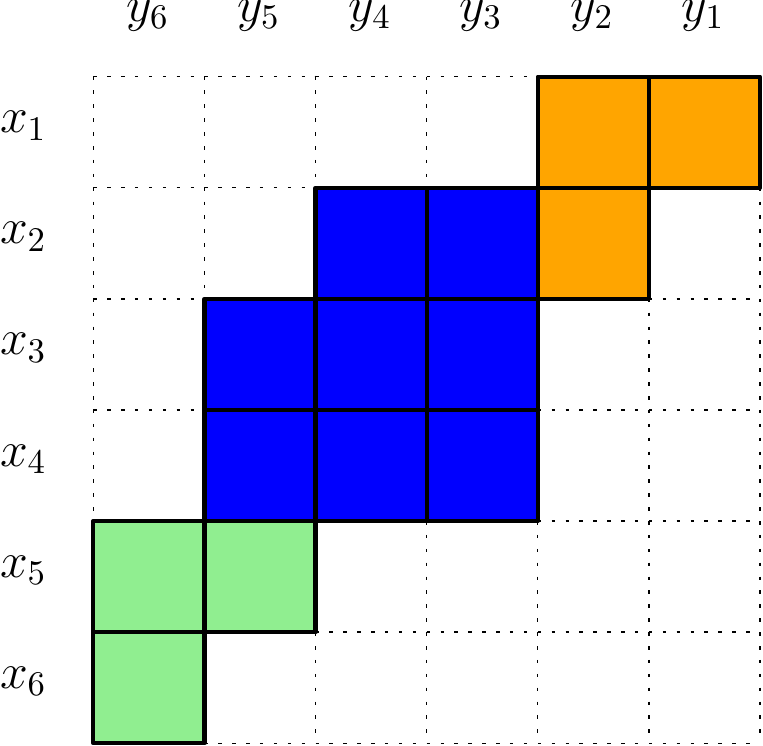}  
   \caption{Rectangular decomposition of a skew Ferrers diagram} \label{fig3}
 \end{figure} 
\end{center}
\end{exam} 

\begin{lem}  \label{Dirichlet}
If $e\in U$, then   $\im(H[E_e])=1$.
\end{lem} 
\begin{proof}  Let $e=\{x_i, y_j\}$.  Assume on the contrary that there are two   edges  $e_1:=\{x_{u_1}, y_{v_1}\}$ and $e_2:=\{x_{u_2}, y_{v_2}\}$ in $E_{e}$ such that  $\{e_1, e_2\}$ is an induced matching of $H$.  Then $u_1\ne u_2$ and $v_1\ne v_2$.  Without loss of   generality,  we assume $u_1<u_2$. Since $e_1, e_2\in E_{e}$, so $u_1<u_2\le i$, $v_1, v_2\le j$,  $\{x_{u_1},y_{v_1}\}\cap N_H(\{x_{i}, y_j\})\ne \emptyset$ and  $\{x_{u_2},y_{v_2}\}\cap N_H(\{x_{i}, y_j\})\ne \emptyset$.  In order to prove this lemma  we divided the proof in two following cases:
 
 \medskip

\noindent {\it Case 1.}  $x_{u_1}, x_{u_2}\in N_H(y_j)$; or $y_{v_1}, y_{v_2}\in N_H(x_i)$.  We will  prove for the case $y_{v_1}, y_{v_2}\in N_H(x_i)$, the remain case will be proved  similarly. Since $\{x_i, y_{v_1}\}$ and $\{x_{u_1},y_{v_1}\}$ are two edges of $H$, by Lemmas \ref{neibour}(1) and \ref{def-ini}(1),  $\{x_{u_1}, x_{u_1+1},\ldots, x_{i}\} \subseteq N_H(y_{v_1})$. Thus  $\{x_{u_2},y_{v_1}\}\in E(H)$, a contradiction.

\medskip
 
\noindent {\it Case 2.} $x_{u_1}\in  N_H(y_j), y_{v_2}\in N_H(x_i)$; or    $x_{u_2}\in  N_H(y_j), y_{v_1}\in N_H(x_i)$. We  will  prove for the first case, and the remain case will be proved similarly. If $v_2<v_1$, then $\{x_{u_2},y_{v_1}\}\in E(H)$ because $\mu_{u_2}\le \mu_{u_1}$. 
 It implies that $\{e_1,e_2\}$ is not an  induced matching of $H$, a contradiction. Therefore, $v_2>v_1$, and thus in this case $v_1<v_2\le j$. Recall $\{x_{u_1},y_{v_1}\}, \{x_{u_1},y_{j}\}\in E(H)$. By     Lemmas \ref{neibour}(1) and \ref{def-ini}(1), $\{y_{v_1},y_{v_1+1},\ldots, y_{j}\}\subseteq N_H(x_{u_1})$, and so $\{x_{u_1},y_{v_2}\}\in E(H)$, a contradiction. 
  \end{proof}

Applying   the Algorithm \ref{Algorithm} for  an initial-closed graph $H$, we get

\begin{thm} \label{Betti} Let $H$ be an initial-closed graph. Then    
$$\beta_{2(n-1)-|U|-|S|,2(n-1)-|S|}(H)\ne 0.$$
In particular,  $\reg(H) = \im(H)= |U|$, and   $\pd(H)\ge 2(n-1)-(|U|+|S|)$.
\end{thm} 
\begin{proof}  Since $|V(H)\backslash S| = 2(n-1)-|S|$, by Hochster formula, we have 
$$ \beta_{2(n-1)-|U|-|S|,2(n-1)-|S|}(H)  \ge   \dim_{\kk} \widetilde{H}_{|U|-1}(\Delta(H\backslash S),\kk)=1. $$
Thus, $\pd(H)\ge 2(n-1)-|U|-|S|$.

We claim that  $\im (H) = |U|$. Indeed, if $U$ is not a maximum induced matching of $H$, then  there is an induced matching $U'$ such that $|U'|>|U|$. Then by Pigenhole principle, there are  at least two edges $e, e'$ in $U'$ belonging to $E_{e_i}$ for some $i$, but this fact contradicts Lemma \ref{Dirichlet}, as claimed. So $\im (H)\le \reg(H)$.

Therefore, by the definition of regularity, we have    $\im(H)\le \reg(H)$. Conversely, in order to prove $\im(H)\ge \reg(H)$, we let  $r:=\reg(H)$. Then there exists $i$ such that $\beta_{i,i+r}(H)\ne 0$. By Hochster formula, there is a subset $W\subseteq V(H)$ such that $|W| = i+r$ and $\dim_{\kk} \widetilde{H}_{r-1}(\Delta(H[W]);\kk) \ne 0$.  Note that $H[W]$ is a  skew Ferrers graph. Applying  the Algorithm \ref{Algorithm} for $H[W]$,  there is an induced matching $U_{W}$ of $H[W]$ such that  $\dim_{\kk} \widetilde{H}_{r-1}(\Delta(H[W]);\kk)=1$, and $r-1=|U_{W}|-1$.   Therefore, $\reg(H) = |U_{W}|\le \im(H)$. We conclude that $\reg(H) = \im(H)=|U|$.
 \end{proof}
 
 \begin{rem} The equality $\reg(H) = \im(H)$ could be obtained also using \cite[Corollary 2.4]{EZ}.
 \end{rem}

\begin{cor} \label{e-Betti2} Let $H$ is an initial-closed graph. If $0<\mu_s=\ldots=\mu_1$, then   $\beta_{2(n-1)-(\mu_1+s),2(n-1)-(\mu_1+s)+2}(H)$   is the  unique extremal Betti number. 
\end{cor} 
\begin{proof}   By   Algorithm \ref{Algorithm}, we have  $ S=\{y_1,\ldots,y_{s-1}, y_{n-\mu_1},\ldots,y_{n-2}\}$,  and   $$U=\{\{x_{s},y_{n-1-\mu_1}\}, \{x_{n-1},y_{n-1}\}\}.$$ 
Thus,  by Theorem \ref{Betti}, $\reg(H)=2$ and  $\beta_{p,p+2}(H)\ne 0$, where $p:=2(n-1)-(\mu_1+s)$. Furthermore, by Proposition \ref{equal_mu}, $\pd(H)=p$.  Therefore, $\beta_{p,p+2}(H)$ is the unique  extremal Betti number. 
\end{proof}

\begin{cor} \label{e-Betti} Let $H$ is an initial-closed graph. If $0<\mu_s<\ldots<\mu_1< n-s$, then   $\beta_{2(n-1)-(\mu_s+s),2(n-1)-(\mu_s+s)+3}(H)$  is the unique extremal Betti number. 
\end{cor} 
\begin{proof}    
By Algorithm \ref{Algorithm}, we have $S=\{y_1,\ldots,y_{s-2}, y_{n-\mu_s},\ldots,y_{n-2}\}$ and $$U=\{\{x_{s-1},y_{s-1}\}, \{x_s, y_{n-1-\mu_s}\}, \{x_{n-1},y_{n-1}\}\}.$$
By Theorem \ref{Betti}, $\reg(H)=3$ and $\beta_{p,p+3}(H)\ne 0$, where $p:=2(n-1)-(\mu_s+s)$. Therefore, the corollary follows from  by Theorem \ref{main}. 
\end{proof}

\noindent {\bf Question 1.}  If  $H$ is an initial-closed graph, then does $R/I(H)$ have the unique   extremal Betti number? 

 \section{Extremal Betti numbers of closed graphs}

We obtain here a partial answer of Ene, Herzog and Hibi conjecture \cite{EHH} in Theorem \ref{main-thm3}.

\begin{lem} \label{main-lem1} Let  $G$ be  a connected closed graph without cut point,  $\mu(G)=(\mu_1,\ldots,\mu_n)$ and $s:=\min\{k-1\mid \mu_k=0\}$. Let $R:=\kk[x_i,y_i\mid i\in V(G)]$.  If one of three following conditions is satisfied:
\begin{enumerate}
\item $ s=0$, or 
\item $0<\mu_s=\ldots=\mu_1$, or 
\item $0<\mu_s<\ldots<\mu_1<n-s$,
\end{enumerate}
then $R/\iin(J_G)$ and $R/J_G$  have an unique extremal Betti number, and they are equal. In particular, $\pd(R/J_G)=\pd(R/\iin(J_G))$ and in   the case   $s=0 $, $\pd(R/\iin(J_G))=n-1$; and in  the remain two cases, $\pd(R/\iin(J_G))=2n-\mu_s-s-2$. 
\end{lem}
\begin{proof}  
Let $H$ be the initial-closed graph associated to $G$, i.e. $I(H)=\iin(J_G)$.

\medskip

(1)   By Proposition \ref{CM}, $\reg(H)=1$ and $\pd(H) = n-1$. We have $R/I(H)$ and  $R/\iin(J_G)$ have the same Betti numbers; so $\reg(R/\iin(J_G))=1$, $\pd(R/\iin(J_G))=n-1$ and  $R/\iin(J_G)$ has an unique extremal Betti number. As $R/\iin(J_G)$ and $R/J_G$ have the same Hilbert function,  they have the same Betti numbers.

\medskip

(2)  Let $p:=2(n-1)-(\mu_1+s)$.  By Corollary  \ref{e-Betti2},   $\beta_{p,p+2}(R/\iin(J_G))$ is the unique extremal Betti number of $R/\iin(J_G)$.  So, $\pd(R/\iin(J_G))=p$ and $\reg(R/\iin(J_G))=2$.  Since $R/J_G$ and $R/\iin(J_G)$ have the same Hilbert function, $\beta_{p,p+2}(R/J_G)= \beta_{p,p+2}(R/\iin(J_G))$.  Thus $\pd(R/J_G)=p$ and $\reg(R/J_G)=2$ and we conclude that $\beta_{p,p+2}(R/J_G)$ is the unique extremal Betti number of $R/J_G$.

\medskip

(3)  Let $p:=2n-\mu_s-s-2$.  By Corollary  \ref{e-Betti},   $\beta_{p,p+3}(R/\iin(J_G))$ is the unique extremal Betti number of $R/\iin(J_G)$; and we obtain the claim as in (2).

\end{proof}

Using Proposition \ref{depth-extremal} and Lemma \ref{main-lem1} we obtain the following theorem:
\begin{thm} \label{main-thm3} 
Let $G$ be a connected closed graph with $\ell$ cut points $v_1,\ldots,v_{\ell}$. Assume that $G=G_1\cup\ldots\cup  G_{\ell+1}$ such that $G_i\cap G_{i+1}=\{v_i\}$ and $G_i\cap G_j = \emptyset$ for $i=1,\ldots,\ell$ and $i\ne j\ne i+1$. Let  $n_i=|V(G_i)|$ and  $\mu(G_i)= (\mu_{i1},\ldots,\mu_{in_i})$, where $s_i:=\min\{k-1\mid \mu_{ik} =0\}$. If for each $i$,    one of three following conditions is satisfied:
\begin{enumerate}
\item $s_{i}=0$,   or 
\item $0<\mu_{is_i}=\ldots=\mu_{i1}$,   or 
\item $0<\mu_{is_i}<\ldots<\mu_{i1}<n_i-s_i$;
\end{enumerate}
 then $R/\iin(J_G)$ and $R/J_G$  have an unique extremal Betti number, and they are equal. In particular, $\pd(R/J_G)=\pd(R/\iin(J_G))$.
 \end{thm}

\begin{rem} If Question 1 has an affirmative answer,   theorem \ref{main-thm3} will be true for all closed graph $G$;  and thus    $\reg(R/J_G)=\reg(R/\iin (J_G))$ and $\pd(R/J_G)=\pd(R/\iin (J_G))$.
\end{rem}

\subsection*{Acknowledgment}  This work was finished during the second author's postdoctoral fellowship at Universidad Nacional Aut\'onoma de Mexico (UNAM) and Universidad Aut\'onoma de Zacatecas  (UAZ). He would like to thank UNAM and UAZ for the financial support and hospitality. He is also partially supported by the NAFOSTED (Vietnam) under grant number 101.04-2015.02. The first author would like to thank Luis Manuel Rivera for stimulating discussions.

  \end{document}